\documentclass[10pt]{article}
\usepackage[utf8]{inputenc}
\usepackage[alphabetic]{amsrefs}
\usepackage{amssymb}
\usepackage{bbm}
\usepackage{xcolor}
\usepackage{hyperref}

\hypersetup{colorlinks}
\usepackage{amsthm}
\usepackage{enumitem}
\usepackage{amsmath}
\allowdisplaybreaks

\usepackage{tikz}
\usetikzlibrary{trees,graphs}
\usetikzlibrary{arrows,arrows.meta,decorations.markings,patterns,calc,shapes.geometric}

\usepackage{ifthen}

\usepackage[export]{adjustbox}

\theoremstyle{plain}
\newtheorem{theorem}{Theorem}[section]
\newtheorem{lemma}[theorem]{Lemma}
\newtheorem{proposition}[theorem]{Proposition}

\newtheorem{corollary}[theorem]{Corollary}

\theoremstyle{definition}
\newtheorem{definition}[theorem]{Definition}

\newtheorem{question}[theorem]{Question}

\theoremstyle{remark}
\newtheorem{remark}[theorem]{Remark}
\newtheorem{construction}[theorem]{Construction}

\newtheorem*{To show}{To show}

\title{Critical exponents of boomerang subgroups in the free group}
\author{Waltraud Lederle}
\date{\today}

\let\epsilon\varepsilon
\renewcommand\emptyset\varnothing
\let\phi\varphi

\newcommand{\tree}{\mathcal T}

\newcommand{\Sub}{\mathrm{Sub}}

\newcommand{\SL}{\mathrm{SL}}

\newcommand{\Iso}{\mathrm{Iso}}

\newcommand{\Sch}{\mathrm{Sch}}
\newcommand{\Core}{\mathrm{C}}

\newcommand{\Gr}{\mathrm{G}}

\newcommand{\Boom}{\mathrm{Boom}}

\newcommand{\N}{\mathbb N}

\newcommand{\frakg}{\mathfrak{g}}

\begin{document}

	\maketitle
	\begin{abstract}
		We construct, for the free group acting on its Cayley tree, boomerang subgroups whose critical exponent is arbitrarily close to the critical exponent of a given finitely generated subgroup.
	\end{abstract}

\section{Introduction}

Let $(X,d_X)$ be a proper geodesic Gromov hyperbolic metric space and $\Iso(X)$ its group of isometries. For a discrete subgroup $\Gamma \leq \Iso(X)$, the \emph{critical exponent} of $\Gamma$ is defined by
$
\delta(\Gamma) := \inf\big\{s > 0 \, \big| \, \sum_{\gamma \in \Gamma} e^{-s \cdot d_X(x,\gamma(x))} < \infty \big\}, 
$
it is independent of the choice of $x \in X$.
Gekhtman--Levit gave the following lower bound for critical exponents of invariant random subgroups.

\begin{theorem}[{\cite[Thm. 1.1]{GekhtmanLevit2019}}]\label{thm:GekhtmanLevit}
	Let $X$ be a proper geodesic Gromov hyperbolic metric space and let $G < \Iso(X)$ be a closed, non-elementary subgroup acting cocompactly on $X$ and admitting a uniform lattice. Let $\mu$ be an invariant random subgroup on $G$ such that $\mu$-almost every subgroup is discrete and infinite. Then, $\mu$-almost every $H \leq G$ satisfies:
	\begin{enumerate}
		\item $\delta(H) > \frac{1}{2} \dim_H(\partial X)$, and
		\item if $H$ is of divergence type, then $\delta(H)=\dim_H(\partial X)$.
	\end{enumerate}
\end{theorem}

In a later paper \cite{GekhtmanLevit2023} they give a similar lower bound on critical exponents of stationary random subgroups.

In \cite{Glasner2017,GlasnerLederle2025} the authors introduced boomerang subgroups of countable groups. For a countable group $\Gamma$, we denote its set of subgroups by $\Sub(\Gamma)$ and endow it with the subspace topology inherited from the obvious embedding $\Sub(\Gamma) \subset \{0,1\}^\Gamma$. With this topology $\Sub(\Gamma)$ is called the \emph{Chabauty space} of $\Gamma$.
Note that the set of finitely generated subgroups is dense in the Chabauty space, because for every $H \in \Sub(\Gamma)$ we can enumerate its elements as $H = \{h_1,h_2,h_3,\dots\}$ and then $H = \lim_{n \to \infty} \langle h_1,\dots,h_n \rangle$.
The group $\Gamma$ acts on the compact, metrizable, totally disconnected space $\Sub(\Gamma)$ continuously via conjugation.
A \emph{boomerang subgroup} of $\Gamma$ is a subgroup $\Delta \in \Sub(\Gamma)$ such that, for every $g \in \Gamma$, there exists a sequence $(n_k) \to \infty$ with $\lim_{k \to \infty} g^{n_k} \Delta g^{-n_k} = \Delta$.
An \emph{invariant random subgroup} of $\Gamma$ is a conjugacy-invariant, Borel probability measure $\mu$ on $\Sub(\Gamma)$.
As a simple consequence of the Poincar\'e recurrence theorem, given any invariant random subgroup $\mu$ of a countable group $\Gamma$, $\mu$-almost every subgroup is a boomerang subgroup of $\Gamma$.
This motivates the question which properties that hold almost surely for invariant random subgroups remain true for boomerang subgroups. 
In this note, we show that the lower bound on the critical exponent does not.

\begin{theorem}\label{thm:main}
	Let $F_d$ be the free group of rank $d$ acting on its Cayley tree.
	Let $\Gamma \leq F_d$ be a finitely generated subgroup, $O \subset \Sub(F_d)$ a neighborhood of $\Gamma$ and $\epsilon > 0$.
	Then, there exists a boomerang subgroup $\Delta \in \Sub(F_d)$ with $\Delta \in O$ and $\delta(\Gamma) \leq \delta(\Delta) \leq \delta(\Gamma) + \epsilon$.
\end{theorem}

The proof uses the visualization of boomerang subgroups via Schreier graphs introduced in \cite[Section 2]{GlasnerLederle2025}. 

The critical exponent is tightly related to the \emph{growth rate} of a subgroup.
Consider again $F_d$, the free group of rank $d$, acting on its Cayley tree $\tree$.
We denote by $B_\tree(1,R)$ the subgraph that is the closed ball of radius $R$ centered at $1$ in $\tree$.
For $g \in F_d$ we denote by $|g|$ the usual word length.
We denote
\[
c_\Gamma(R) := |\Gamma \cap B_\tree(1,R)| = \{\gamma \in \Gamma \mid |\gamma| \leq R\}.
\]
An alternate description of the critical exponent is
\[
\delta(\Gamma) = \lim_{R \to \infty} \frac{\ln(c_\Gamma(R))}{R},
\]
where the limit exists by Fekete's lemma.
Exponentiating we get $b(\Gamma) := e^{\delta(\Gamma)}$, and $b(\Gamma)$ measures the \emph{growth rate} in the sense that $c_\Gamma(R) \sim b(\Gamma)^R$.

In a recent sequence of papers \cite{LouvarisWiseYehuda2024,LouvarisWiseYehuda2025}, Louvaris, Wise and Yehuda showed that every number in $[1,2d-1]$ can be realized as the growth rate of a subgroup of $F_d$, and the set of growth rates of finitely generated subgroups of $F_d$ is dense. This now immediately gives the following corollary.

\begin{corollary}\label{cor:cedense}
	Let $\Boom(F_d) \subset \Sub(F_d)$ denote the set of boomerang subgroups.
	Then, $\delta(\Boom(F_d))$ is dense in $[0,\ln(2d-1)]$, and $b(\Boom(F_d))$ is dense in $[1,2d-1]$.
\end{corollary}

In some sense this corollary means that there are boomerang subgroups that cannot be ``typical instances" of invariant random subgroups, or a given stationary random subgroup.

A subgroup $\Delta \leq F_d$ is \emph{confined} if there is no sequence $(g_n)$ in $F_d$ with $\lim_{n\to \infty} g_n \Delta g_n^{-1} = \{1\}$. 

\begin{corollary}\label{cor:confined}
	The set of critical exponents of non-confined boomerang subgroups is dense in $[0,\ln(2d-1)]$.
\end{corollary}

Fraczyk \cite[Corollary 3.2]{Fraczyk2019} showed that subgroups of $F_d$ with critical exponent $\leq \ln(2d-1)/2$ cannot be confined.
Lower semicontinuity of the critical exponent (see Proposition \ref{prop:liminf}), together with Theorem \ref{thm:GekhtmanLevit} now immediately implies the following.

\begin{corollary}\label{cor:irs supported}
	There exist boomerang subgroups $\Delta < F_d$ such that the only invariant random subgroup supported on the closure of the conjugacy class of $\Delta$ is the Dirac measure at the trivial group.
\end{corollary}

\begin{remark}
Using results from \cite{GekhtmanLevit2023} about critical exponents one can derive a similar statement involving stationary random subgroups.
But still \cite[Theorem 1.8]{GlasnerLederle2025} implies that nontrivial boomerang subgroups of the free group are geometrically dense.
\end{remark}

Corollary \ref{cor:irs supported} is another stark contrast to lattices of higher $\mathbb{Q}$-rank in simple Lie groups, like $\SL_n(\mathbb{Z})$ with $n \geq 3$.
The main theorem of \cite{GlasnerLederle2025} shows that
every non-trivial boomerang subgroup there is finite and central or has finite index, in particular has only finitely many conjugates and is in particular confined.

An underlying philosophical question that motivated this work is: ``How are boomerang subgroups more general than typical instances of invariant random subgroups?"
In this work, we look at the question from the geometric point of view of the critical exponent.
In joint work with Yair Glasner and Tobias Hartnick \cite{GlasnerHartnickLederle2025}, we consider a completely different angle: non-singular dynamics and orbit equivalence. We start from the observation that whenever $\mu$ is a Borel probability measure on $\Sub(\Gamma)$ such that each $\gamma \in \Gamma$ acts on $(\Sub(\Gamma),\mu)$ as a conservative transformation, then $\mu$-almost every subgroup is a boomerang subgroup.
We then construct uncountably many such measures on $\Sub(F_d)$ that are singular with respect to every invariant random subgroup.
Whether there is a connection to this work is unclear, see Question \ref{qu:random}.

\paragraph*{Acknowledgements.} I thank Tianyi Zheng for asking me whether boomerang subgroups can have small critical exponents.
I thank Ilya Gekhtman for helpful clarifications and Yair Glasner and Tobias Hartnick for helpful discussions and feedback.
Part of this work was done on the LG\&TBQ2 conference in Montreal, my travel there was partially funded by Thomas Koberda's Shannon Fellowship.

\section{Schreier graphs and their cores}

\subsection{Schreier graphs}

The subgroups of $F_d = \langle a_1,\dots,a_d \rangle$ are in 1-1 correspondence with (isomorphism classes of) directed, labeled, rooted graphs, where each vertex has exactly one ingoing and one outgoing edge labeled $a_i$ for each $i = 1,\dots,d$. An edge can be ingoing and outgoing for the same vertex, it is then a loop.
 We call such a graph a \emph{Schreier graph}.
We can write any $g \in F_d$ uniquely as reduced word $g = b_1 \dots b_k \in F_d$, where $b_j \in \{a_1,\dots,a_d,a_1^{-1},\dots,a_d^{-1}\}$.
We can interpret $g$ as a path\footnote{Our paths and cycles are allowed to visit vertices and edges multiple times, and go against the direction of an edge, but have no backtracking unless explicitly allowed to.}
 $\omega_g$ in the Schreier graph starting at the root and walking along the edges with label $b_1,\dots,b_k$, where ``walking along $a_i^{-1}$" means ``walking along the directed edge with label $a_i$ into the opposite direction". In addition $g$ determines a vertex $v_g$ that is the endpoint of that path.
The group $F_d$ acts on the set of Schreier graphs by changing the root from $v_1$ to $v_g$.

Given a Schreier graph, the corresponding subgroup of $F_d$ is
$\{\gamma \in F_d \mid v_\gamma = v_1\}$, or all $\gamma \in F_d$ leaving the root fixed.
Conversely, if $\tree$ is the (directed, labeled, rooted) Cayley tree of $F_d$ and $\Gamma \leq F_d$, then the \emph{Schreier graph} of $\Gamma$ is the quotient $\Sch(\Gamma) = \Gamma {\setminus} \tree$.
This correspondence between Schreier graphs and subgroups is equivariant, with respect to the action of $F_d$ on Schreier graphs (via root change) and on subgroups (via conjugation).

The topology on $\Sub(F_d)$ can be formulated via the Schreier graphs as follows.
For a graph $\frakg$ and a vertex $v$, we denote by $B_\frakg(v,R)$ the subgraph induced by all vertices of distance at most $R$ from $v$.
Given $\Gamma \leq \Sub(F_d)$, a neighborhood basis of $\Gamma$ is 
$\{\Delta \in \Sub(F_d) \mid B_{\Sch(\Gamma)}(v_1,R) \cong B_{\Sch(\Delta)}(v_1,R) \}$,
where $R$ runs over all natural numbers and $\cong$ denotes isomorphism as directed, rooted, labeled, directed graphs.
We hence can determine boomerang subgroups by their Schreier graphs, as described in the following proposition.

\begin{proposition}[{\cite[Section 2]{GlasnerLederle2025}}]\label{prop:boomerangschreier}
	A subgroup $\Delta \leq F_d$ is a boomerang subgroup if and only if for every $g \in F_d$ and every $R > 0$ there exists an $n > 0$ such that the balls $B_{\Sch(\Gamma)}(v_1,R)$ and $B_{\Sch(\Gamma)}(v_{g^n},R)$ are isomorphic as labeled, rooted graphs with root $v_1$ and $v_{g^n}$, respectively.
\end{proposition}

\subsection{Core graphs}

Given a Schreier graph, it is only the part of the graph that contains all cycles that determines the corresponding subgroup. This part is called the \emph{core}.
For a comprehensive treatment of the core we recommend \cite{KapovichMyasnikov2002}.
%

\begin{definition}
	Let $\Gamma \leq F_d$ be a subgroup and $\Sch(\Gamma)$ its Schreier graph. 
	The \emph{core graph} of $\Gamma$ is the unique minimal connected (directed, labeled, rooted) subgraph $\Core(\Gamma)$ of $\Sch(\Gamma)$ containing all cycles based at the root.
\end{definition}

In particular, all vertices have valency at least two, except possibly the root.
From $\Core(\Gamma)$ we can recover $\Sch(\Gamma)$ by hanging infinite trees at all vertices of valency less than $2d$.
Thus also (isomorphism classes of) core graphs are in 1-1 correspondence with subgroups of $F_d$.
Given a core graph $\frakg$, denote $\Gr(\frakg) \leq F_d$ the corresponding subgroup, so $\Core(\Gr(\frakg))=\frakg$ and $\Gr(\Core(\Gamma))=\Gamma$.

\begin{remark}\label{rem:finite core}
	It is not difficult to see that $\Gamma$ is finitely generated if and only if $\Core(\Gamma)$ is finite. 
	More precisely, let $\Gamma = \langle \gamma_1,\dots,\gamma_n \rangle$, then
	$\Core(\Gamma) = \bigcup_{i=1}^n \omega_{\gamma_i} \subset \Sch(\Gamma)$.
	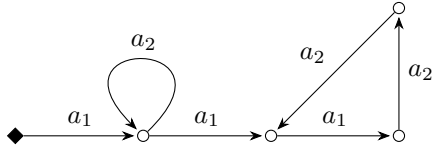
\begin{figure}
		\centering
\begin{tikzpicture}[shorten >=1pt,node distance=1.7cm,auto]
		\node[draw,diamond,inner sep=1.5pt,fill=black]        (A)              {};
		\node[draw,circle,inner sep=1.5pt]        (B) [right of=A] {};
		\node[draw,circle,inner sep=1.5pt]        (C) [right of=B] {};
		\node[draw,circle,inner sep=1.5pt]        (D) [right of=C] {};
		\node[draw,circle,inner sep=1.5pt]        (E) [above of=D] {};
		\draw[->, >=Stealth] (A) edge node {$a_1$} (B)
			  (B) edge [out=44,in=135, looseness=40] node [swap] {$a_2$} (B)
			  (B) edge node {$a_1$} (C)
			  (C) edge node {$a_1$} (D)
			  (D) edge node [swap] {$a_2$} (E)
			  (E) edge node [swap] {$a_2$} (C);
\end{tikzpicture}
		\caption{The graph $\Core(\langle a_1 a_2 a_1^{-1}, a_1^3 a_2^2 a_1^{-2} \rangle)$.}
		\label{fig:core}
	\end{figure}
	Moreover, in this case, $\Gamma$ is a cocompact lattice in the automorphism group of the universal cover of $\Core(\Gamma)$, which is the unique minimal subtree of the Cayley tree $\tree$ of $F_d$ containing the verices $\Gamma \subset \tree$. This minimal subtree has more than two ends if and only if $\Gamma$ is not cyclic.
\end{remark}

\begin{lemma}\label{lem:powersincore}
	Let $\Gamma \leq F_d$ be finitely generated and $g \in F_d$. The following are equivalent.
	\begin{enumerate}
		\item There exists $n \geq 1$ with $g^n \in \Gamma$.
		\item The set of vertices $\{v_{g^n} \mid n \in \mathbb{Z}\} \subset \Sch(\Gamma)$ is finite.
	\end{enumerate}
\end{lemma}

Note that $\Core(\Delta) = \Sch(\Delta)$ for every non-trivial boomerang subgroup $\Delta \leq F_d$. In particular the core of a boomerang subgroup is $2d$-regular. 
This follows from \cite[Theorem 1.8]{GlasnerLederle2025}, but can also be seen directly.
Because assume otherwise; so there exists a vertex $v_g$ that is in the Schreier graph but not in its core. Upon making $g$ longer if necessary we can assume that $g$ is cyclically reduced and then $v_{g^n} \notin \Core(\Delta)$ for any $n > 0$. In particular, for every $n > 0$, the large ball $B_{\Sch(\Delta)}(v_{g^n},|g^{n-1}|)$ is a tree, and therefore $\Delta$ cannot satisfy Proposition \ref{prop:boomerangschreier}. In fact, $\lim_{n \to \infty} g^n \Delta g^{-n} = \{1\}$.

In what follows, we introduce a notation of gluing together two cores along a word in $F_d$.

\begin{definition}
	Let $\frakg_1$ and $\frakg_2$ be finite core graphs.
	Let $g \in F_d$ and 
	write it as reduced word $g = b_1 \dots b_k$.
	Let $k_1 := \max\{1 \leq i \leq k \mid v_{b_1 \dots b_i} \in \frakg_1\}$
	and $k_2 := \max\{1 \leq i \leq k \mid v_{b_k^{-1}} \dots v_{b_{k-i}^{-1}} \in \frakg_2\}$.
	We call $g$ an \emph{admissible connector} for $\frakg_1$ with $\frakg_2$ if
	 $k_1 + k_2 < k$.
	 Then, we call the subword
	 $b_1 \dots b_{k_1}$ the \emph{initial segment of $g$},
	 the word $b_k^{-1} \dots b_{k-k_2}^{-1}$ the \emph{terminal segment of $g$},
	 and $b_{k_1 + 1} \dots b_{k - k_2 - 1}$ the \emph{join segment}.
\end{definition}

\begin{definition}\label{def:glueing}
	Let $\frakg_1$ and $\frakg_2$ be finite core graphs.
	Let $g = b_1 \dots b_k \in F_d$ be an admissible connector.
	Then, the core graph
	\[
	\frakg_1 \sqcup_g \frakg_2
	\]
	is constructed as follows: 
	We start from the disjoint union $\frakg_1 \sqcup \frakg_2$,
	and we take a path of length $k - (k_1+k_2)$ labeled by the join segment. We identify the starting vertex of this path with the vertex corresponding to the initial segment $v_{b_1 \dots b_{k_1}} \in \frakg_1$, and its ending vertex with the vertex corresponding to the terminal segment $v_{b_k^{-1}} \dots v_{b_{k-k_2}^{-1}} \in \frakg_2$.
	We declare the root of the resulting graph to be the root of $\frakg_1$, and then we remove leafs until we are left with a core graph.
	\begin{figure}
		\centering
		\includegraphics[scale=0.45]{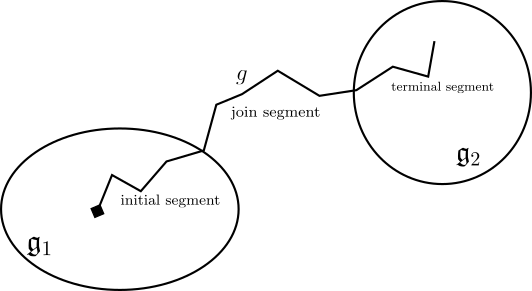}
		\caption{Depiction of $\frakg_1 \sqcup_g \frakg_2$}
		\label{fig:cup}
	\end{figure}
\end{definition}


\begin{lemma}\label{lem:glueing}
	In the situation of Definition \ref{def:glueing}, let $\Gamma_i = \Gr(\frakg_i) \leq F_d$, for $i=1,2$.
	Then, $\langle \Gamma_1, g \Gamma_2 g^{-1} \rangle= \Gamma_1
	\ast g \Gamma_2 g^{-1}$ and
	$$\frakg_1 \sqcup_g \frakg_2 = \Core(\langle \Gamma_1, g \Gamma_2 g^{-1} \rangle).$$	
\end{lemma}

\begin{proof}
	See \cite[Proposition 7.9]{KapovichMyasnikov2002} and its proof.
\end{proof}

Using this we can construct boomerang subgroups containing a given finitely generated group.

\begin{construction}\label{par:constr}
Let  $\Gamma \leq F_d$ be finitely generated.
If $\Gamma$ has finite index or is the trivial subgroup, there is nothing to do.
Otherwise, enumerate $F_d = \{g_1,g_2,\dots\}$.
Set $\Gamma_0 := \Gamma$.
Assume $\Gamma_{n-1}$ has been constructed for $n \geq 1$; it is finitely generated hence $\Core(\Gamma_{n-1})$ is finite.
If $g_n$ satisfies the equivalent conditions of Lemma \ref{lem:powersincore},
we can set $\Gamma_n := \Gamma_{n-1}$.
Otherwise, let $R_{n-1} = \min\{R \mid \Core(\Gamma_{n-1}) \subset B_{\Sch(\Gamma_{n-1})}(v_1,R) \}$ be the radius of $\Core(\Gamma_{n-1})$ with center $v_1$.
There exists $M_n > 0$ such that, for all $m_n \geq M_n$,
\begin{itemize}
	\item the requirements of Definition \ref{def:glueing} are satisfied with $\frakg_1 = \frakg_2 = \Core(\Gamma_{n-1})$ and $g = g_n^{m_n}$, and
	\item $|g_n^{m_n}| > 2R_{n-1}$.
\end{itemize}
So choose any $m_n \geq M_n$.
Set
$$\Gamma_n := \langle \Gamma_{n-1}, g_n^{m_n} \Gamma_{n-1} g_n^{-m_n} \rangle = \Gamma_{n-1} \ast g_n^{m_n} \Gamma_{n-1} g_n^{-m_n}.$$
By Lemma \ref{lem:glueing} we know that $\Core(\Gamma_n) = \Core(\Gamma_{n-1}) \sqcup_{g_n^{m_n}} \Core(\Gamma_{n-1})$.
In particular, for all $n' \geq n$, we have $B_{\Sch(\Gamma_{n'})}(v_1,R_{n-1}) \cong B_{\Sch(\Gamma_{n'})}(v_{g_n^{m_n}},R_{n-1})$.
By construction, we have $\Gamma \leq \Gamma_1 \leq \Gamma_2 \leq \dots$.
Set
$$\Delta := \bigcup_n \Gamma_n.$$
Also $\Core(\Delta) = \bigcup_n \Core(\Gamma_n)$ is an increasing union.
Now we see from Proposition~\ref{prop:boomerangschreier} that $\Delta$ is a boomerang subgroup in $F_d$, because for every non-trivial $g \in F_d$ and every $k \geq 1$ there are infinitely many $m \in \mathbb{N}$ with $g^m \in \{g_k,g_{k+1},\dots\}$.
\end{construction}

\begin{question}\label{qu:random}
	In joint work with Glasner and Hartnick, we are constructing elementwise conservative actions of the free group. Can Construction \ref{par:constr} be randomized, so that we obtain a  probability measure $\mu$ on $\Sub(F_d)$ such that the conjugation action on $(\Sub(F_d),\mu)$ is non-singular and elementwise conservative?
\end{question}

\subsection{Critical exponents and cores}

Let $\Gamma \leq F_d$. Recall that the critical exponent of $\Gamma$ is
$\delta(\Gamma) = \lim_{R \to \infty} \frac{\ln(c_\Gamma(R))}{R}$,
where 
\begin{align*}
	c_\Gamma(R) &= |\{\gamma \in \Gamma \mid |\gamma| \leq R\}| \\
	&= |\{g \in F_d \mid \omega_{g} \subset \Core(\Gamma) \text{ is a cycle based at } v_1\}|.
\end{align*}
Hence calculating critical exponents amounts to counting cycles in core graphs.

\begin{lemma}\label{lem:cRestimate}
	Let $\Gamma \leq F_d$ be a finitely generated subgroup. 
	Let $g \in F_d$ be an admissible connector for $\Core(\Gamma)$ with itself,
	and let $J$ be the length of the join segment.
%
	Then
	\[
	c_{\Gamma \ast g \Gamma g^{-1}}(R) \leq 
	\sum_{i=0}^{\left\lfloor \frac{R}{2J} \right\rfloor} \sum_{\substack{\alpha \in \N^{2i+1} \\ \sum_j \alpha_j = R + 2i(|g|-2J)}
	} \prod_{j=0}^{2i} c_\Gamma(\alpha_j).
	\]
\end{lemma}

\begin{proof}
	In the graph $\Core(\Gamma) \sqcup_g \Core(\Gamma)$, we call $\frakg$ the copy of $\Core(\Gamma)$ containing the root, and $\frakg'$ the other copy.	
	A cycle based at the root that does not remain in $\frakg$ has to travel the join segment $p$ of $\omega_g$, which joins $\frakg$ and $\frakg'$, back and forth.
	Estimating the number of those cycles of length $R$ hence boils down to seeing how often we can travel this segment back-and-forth during that time, and how often we can loop around inside $\frakg$ and $\frakg'$ in between those travels.
	
	The number of times the path between $\frakg$ and $\frakg'$ can be traveled back-and-forth by a cycle of length $R$ is at most $\lfloor \frac{R}{2J} \rfloor$.
	
	Let now  $0 \leq i \leq \lfloor \frac{R}{2J} \rfloor$.
	Consider a cycle in $\Core(\Gamma) \sqcup_g \Core(\Gamma)$ of length at most $R$ based at the root that visits $\frakg'$ exactly $i$ times.
	We can write it as $(p_0,q,p_1,q^{-1},\dots,q^{-1},p_{2i})$, where $p_j$ is a path in $\frakg$ for $j$ even, a path in $\frakg'$ for $j$ odd, and $q$ is the unique path connecting $\frakg$ and $\frakg'$.
	Let $J_1$ be the length of the initial segment and $J_2$ the length of the terminal segment of $g$.
	
	We now enlarge the cycle (possibly introducing backtracking to not change its homotopy type) to write it as concatenation of a cycle in $\frakg$, followed by $\omega_g$, followed by a cycle in $\frakg'$, followed by $\omega_{g}$ backwards, etc.
	For this, we need to add edges as follows. To $p_0$ we add at most $J_1$ many edges at the end (so that it ends at the root), and to $p_{2i}$ we add at most $J_1$ many edges at the start (so that it starts at the root).
	Then, to each $q$ we add $J_1$ edges at the start and $J_2$ edges at the end to make it into $\omega_g$; and analogously we add $J_1 + J_2$ many edges to $q^{-1}$.
	Lastly, for each $0 < j < 2i$, if $j$ is even we enlarge $p_j$ by at most $2J_1$ many edges and if $j$ is odd we enlarge $p_j$ by at most $2J_2$ many edges.
	We obtained now a closed path 
	$(c_0,\omega_g,c_1,\omega_g^{-1},c_2,\omega_g,\dots,\omega_g^{-1},c_{2i})$ with backtracking,
	where each $c_j$ is a closed path that is based at the root $v_1$ if $j$ is even and based at $v_g$ if $j$ is odd.
	This closed path actually lives in $\Sch(\langle \Gamma, g \Gamma g^{-1}\rangle)$, since it might happen that $\omega_g$ is not contained in $C(\Gamma) \sqcup_g C(\Gamma)$, but this does not influence the rest of the argument.
	Using $|g| = J_1 + J + J_2$,
	the total length of this closed path is at most $R + 4i(|g|-J)$, of which $2i|g|$ is spent on the paths $g$ or $g^{-1}$.
	
	Then, the number of cycles that visits the second copy of $\frakg$ exactly $i$ times is smaller than or equal to
	\[
	\sum_{\substack{\alpha \in \N^{2i+1} \\ \sum_j \alpha_j = R + 4i(|g|-J) - 2i|g|}
	} \prod_{j=0}^{2i} c_\Gamma(\alpha_j).
	\]
	This finishes the proof.
\end{proof}

By Remark \ref{rem:finite core}, if $\Gamma \leq F_d$ is finitely generated but not cyclic,
we can apply the following theorem by Coornaert 
to estimate $c_\Gamma(R)$.

\begin{proposition}[{\cite[Th\'eor\`eme 7.1]{Coornaert1993}}]\label{prop:coornaert}
	Let $\Gamma$ be a group acting properly and cocompactly by
	isometries on a proper geodesic Gromov hyperbolic metric space $X$ whose boundary contains more than two points.
	Then there exists $K \geq 1$ such that for every $R \geq 0$
	\[
	\frac{1}{K} e^{\delta(\Gamma) R} \leq c_\Gamma(R) \leq K e^{\delta(\Gamma) R}.
	\]
\end{proposition}

The constant in this proposition depends on $\Gamma$.

\begin{remark}
	For $\Gamma \leq F_d$ we have $\delta(\Gamma) = 0$ if and only if $\Gamma$ is cyclic.
	If $\Gamma$ is the trivial group, then $c_g(R)=1$ for all $R$; otherwise let $\Gamma = \langle g \rangle$ and assume for simplicity that $g$ is cyclically reduced, so $\Core(\Gamma)$ is just a cycle of length $|g|$.
	Then $c_\Gamma(R) = 2 \cdot \lfloor \frac{R}{|g|} \rfloor + 1$.
\end{remark}

Now we can estimate the critical exponent of subgroups of the form $\Gamma \ast g \Gamma g^{-1}$.

\begin{proposition}\label{prop:exponentestimate}
	Let $\Gamma \leq F_d$ be a finitely generated subgroup.
	Let $g \in F_d$ be an admissible connector for $\Core(\Gamma)$ with itself.
	Denote by $J$ the length of the join segment of $g$.
	\begin{enumerate}
		\item Assume $\Gamma$ is not cyclic and that $2J > |g|$.
		Then $$\delta(\Gamma) \leq \delta(\langle \Gamma, g \Gamma g^{-1}\rangle) \leq \delta(\Gamma) + K e^{-\delta(\Gamma)(2J-|g|)},$$ where $K$ is the constant from Proposition \ref{prop:coornaert}.
		\item Assume $\Gamma = \langle \gamma \rangle$ is cyclic and nontrivial. 
		Assume $\gamma$ is cyclically reduced and no prefix of $\gamma$ or $\gamma^{-1}$ is a prefix of $g$, and no suffix of $\gamma$ or $\gamma^{-1}$ is a suffix of $g$.
		Assume in addition that $|g|$ is divisible by $|\gamma|$. Then
		\[
		\frac{1}{|\gamma||g|} \leq \delta(\langle \Gamma, g \Gamma g^{-1}\rangle) \leq \frac{1}{|\gamma|\sqrt{|g|}}.		
		\]
	\end{enumerate}
\end{proposition}

\begin{proof}
	Write $\Gamma' := \langle \Gamma, g \Gamma g^{-1}\rangle \cong \Gamma \ast g \Gamma g^{-1}$.
	
	(1) Plugging Proposition \ref{prop:coornaert} into Lemma \ref{lem:cRestimate} we obtain
	\begin{align*}
		c_{\Gamma'}(R) &\leq \sum_{i=0}^{\left\lfloor \frac{R}{2J} \right\rfloor} \sum_{\substack{\alpha \in \N^{2i+1} \\ \sum_j \alpha_j = R + 2i(|g|-2J)}
		} \prod_{j=0}^{2i} c_\Gamma(\alpha_j) \\
		&\leq \sum_{i=0}^{\left\lfloor \frac{R}{2J} \right\rfloor} \sum_{\substack{\alpha \in \N^{2i+1} \\ \sum_j \alpha_j = R + 2i(|g|-2J)}
		} \prod_{j=0}^{2i} K e^{\delta(\Gamma) \alpha_j} \\
		&= 	\sum_{i=0}^{\left\lfloor \frac{R}{2J} \right\rfloor} K^{2i+1} e^{\delta(\Gamma) (R + 2i(|g|-2J))} \sum_{\substack{\alpha \in \N^{2i+1} \\ \sum_j \alpha_j = R + 2i(|g|-2J)}} 1 \\
		&= e^{\delta(\Gamma)R} \sum_{i=0}^{\left\lfloor \frac{R}{2J} \right\rfloor} K^{2i+1} e^{ 2i \delta(\Gamma)(|g|-2J)} \binom{R + 2i(|g|-2J) + 2i}{2i} \\
		&\leq K e^{\delta(\Gamma)R} \sum_{j=0}^{R} K^{j} e^{ j \cdot \delta(\Gamma)(|g|-2J)} \binom{R}{j} \\
		&= K e^{\delta(\Gamma)R} (1 + Ke^{\delta(\Gamma)(|g|-2J)})^R.
	\end{align*}
	Hence
	\begin{align*}
		\delta(\Gamma') &\leq \lim_{R \to \infty} \frac{\ln(K e^{\delta(\Gamma)R} (1 + Ke^{\delta(\Gamma)(|g|-2J)})^R)}{R} \\
		&= \delta(\Gamma) + \ln(1 + Ke^{\delta(\Gamma)(|g|-2J)}) \\
		&\leq \delta(\Gamma) + Ke^{\delta(\Gamma)(|g|-2J)}.
	\end{align*}
	
	(2) 
	%
	We first assume that $|\gamma|=1$, so without loss of generality $\Gamma = \langle a_1 \rangle$. 
	Let $n = |g|$.
	Consider the polynomial
	\[
	P_n(t) := \begin{cases}
					2 t^n + t - 1 & n \text{ odd} \\
					2 \frac{t^{n}+t}{t+1} - 1 & n \text{ even}.
			  \end{cases}
	\]
	Kwon and Park showed in \cite{KwonPark2018,Kwon2019} that $\delta(\Gamma') = \ln(x_n^{-1})$, where $x_n$ is the unique real root of $P_n(t)$.
	Now we can compute using the intermediate value theorem that
	$ e^{- \frac{1}{\sqrt{n}}} \leq x_n \leq e^{-\frac{1}{n}}$ and hence
	\[
	\frac{1}{n} \leq \delta(\Gamma') \leq \frac{1}{\sqrt{n}}.
	\]
	For the general case,
	using the definition of the critical exponent with the Poincar\'e series, it is easy to see that if we scale the metric by a factor $\lambda > 0$, then the critical exponent of a subgroup gets divided by $\lambda$.
	Hence, if we scale the above situation by the factor $|\gamma|$ we get
	\[
	\frac{1}{|\gamma||g|} \leq \delta(\Gamma') \leq \frac{1}{|\gamma|\sqrt{|g|}}.
	\]
\end{proof}

\section{Proof of the main theorem}

As an important tool we need that the critical exponent is lower semi-continuous.

\begin{proposition}[{\cite[Proposition 10.4]{GekhtmanLevit2023}}]\label{prop:liminf}
	Let $(\Gamma_k)$ be a converging sequence of subgroups and $\lim_{k \to \infty} \Gamma_k = \Gamma$. Then
	$$\liminf_{k \to \infty} \delta(\Gamma_k) \geq \delta(\Gamma).$$
	In particular, if $\Gamma_1 \leq \Gamma_2 \leq \dots$ then $\lim_{k \to \infty} \delta(\Gamma_k) = \delta(\Gamma)$.
\end{proposition}

\begin{proof}[Proof of Theorem \ref{thm:main}]
	If $\Gamma \leq F_d$ has finite index, there is nothing to show.
	Without loss of generality, assume that there exists $R > 0$ such that
	$O = \{\Gamma' \leq F_d \mid B_{\Sch(\Gamma')}(v_1,R) \cong B_{\Sch(\Gamma)}(v_1,R)\}$.
	Also, by adding a suitable generator if necessary using Proposition \ref{prop:exponentestimate}(2), we can assume without loss of generality that $\Gamma$ is not cyclic.

	Enumerate the elements $F_d = \{g_1,g_2,\dots\}$; assume that $\langle g_1 \rangle \cap \Gamma = \{1\}$.	
	We perform Construction \ref{par:constr}, and also use the notation from there.
	Choosing $m_1$ large enough, we can ensure that $\Gamma_n \in O$ for all $n \geq 1$.
	Moreover, for every $n \geq 1$, by Proposition \ref{prop:exponentestimate}, we can ensure that $m_n$ is large enough such that $\delta(\Gamma_n) < \delta(\Gamma_{n-1}) + \frac{\epsilon}{2^n}$.
	Then by Proposition \ref{prop:liminf} we have $\delta(\Delta) < \delta(\Gamma) + \epsilon$.
%
%
%
%

\end{proof}

\begin{proof}[Proof of Corollary \ref{cor:confined}]
	The first part is clear from the above proof.
	The second follows with Theorem \ref{thm:GekhtmanLevit} and Proposition \ref{prop:liminf}.
\end{proof}

\bibliographystyle{alpha}
\bibliography{/home/waldi/Nextcloud/work/references}

\begin{bibdiv}
\begin{biblist}

\bib{Coornaert1993}{article}{
      author={Coornaert, Michel},
       title={{M}esures de {P}atterson-{S}ullivan sur le bord d’un espace
  hyperbolique au sens de {G}romov},
        date={1993-06},
        ISSN={0030-8730},
     journal={Pacific Journal of Mathematics},
      volume={159},
      number={2},
       pages={241\ndash 270},
}

\bib{Fraczyk2019}{article}{
      author={Fraczyk, Mikolaj},
       title={Kesten’s theorem for uniformly recurrent subgroups},
        date={2019-03},
        ISSN={1469-4417},
     journal={Ergodic Theory and Dynamical Systems},
      volume={40},
      number={10},
       pages={2778\ndash 2787},
}

\bib{GlasnerHartnickLederle2025}{unpublished}{
      author={Glasner, Yair},
      author={Hartnick, Tobias},
      author={Lederle, Waltraud},
       title={Elementwise conservative actions and new constructions of
  boomerang subgroups},
        date={2025},
        note={in preparation},
}

\bib{GekhtmanLevit2019}{article}{
      author={Gekhtman, Ilya},
      author={Levit, Arie},
       title={Critical exponents of invariant random subgroups in negative
  curvature},
        date={2019-03},
        ISSN={1420-8970},
     journal={Geometric and Functional Analysis},
      volume={29},
      number={2},
       pages={411\ndash 439},
}

\bib{GekhtmanLevit2023}{misc}{
      author={Gekhtman, Ilya},
      author={Levit, Arie},
       title={Stationary random subgroups in negative curvature},
   publisher={arXiv},
        date={2023},
}

\bib{GlasnerLederle2025}{article}{
      author={Glasner, Yair},
      author={Lederle, Waltraud},
       title={Strong subgroup recurrence and the {N}evo–{S}tuck–{Z}immer
  theorem},
        date={2025-06},
        ISSN={1460-244X},
     journal={Proceedings of the London Mathematical Society},
      volume={130},
      number={6},
}

\bib{Glasner2017}{article}{
      author={Glasner, Yair},
       title={Invariant random subgroups of linear groups},
        date={2017},
        ISSN={0021-2172},
     journal={Israel J. Math.},
      volume={219},
      number={1},
       pages={215\ndash 270},
         url={https://doi.org/10.1007/s11856-017-1479-x},
        note={With an appendix by Tsachik Gelander and Yair Glasner},
      review={\MR{3642021}},
}

\bib{KapovichMyasnikov2002}{article}{
      author={Kapovich, Ilya},
      author={Myasnikov, Alexei},
       title={Stallings foldings and subgroups of free groups},
        date={2002-02},
        ISSN={0021-8693},
     journal={Journal of Algebra},
      volume={248},
      number={2},
       pages={608\ndash 668},
}

\bib{KwonPark2018}{article}{
      author={Kwon, Sanghoon},
      author={Park, Jung-Hyeon},
       title={Ihara zeta function of dumbbell graphs},
        date={2018},
        ISSN={1976-8605,2288-1433},
     journal={Korean J. Math.},
      volume={26},
      number={4},
       pages={741\ndash 746},
         url={https://doi.org/10.11568/kjm.2018.26.4.741},
      review={\MR{3898943}},
}

\bib{Kwon2019}{article}{
      author={Kwon, Sanghoon},
       title={On the set of critical exponents of discrete groups acting on
  regular trees},
        date={2019},
        ISSN={0304-9914},
     journal={J. Korean Math. Soc.},
      volume={56},
      number={2},
       pages={475\ndash 484},
      review={\MR{3922027}},
}

\bib{LouvarisWiseYehuda2024}{misc}{
      author={Louvaris, Michail},
      author={Wise, Daniel~T.},
      author={Yehuda, Gal},
       title={Density of growth-rates of subgroups of a free group and the
  non-backtracking spectrum of the configuration model},
   publisher={arXiv},
        date={2024},
}

\bib{LouvarisWiseYehuda2025}{misc}{
      author={Louvaris, Michail},
      author={Wise, Daniel~T.},
      author={Yehuda, Gal},
       title={Subgroups of a free group with every growth rate},
   publisher={arXiv},
        date={2025},
}

\end{biblist}
\end{bibdiv}
\end{document}